\documentclass{amsart}
\usepackage[english]{babel}
\usepackage{enumerate, tabularx, enumitem,booktabs, amsthm,amsmath,amsfonts,amssymb}
\usepackage{hyperref}

\newtheorem{thm}{Theorem}[section]
\newtheorem{thmintro}{Theorem}
\newtheorem{lemma}[thm]{Lemma}
\newtheorem{corol}[thm]{Corollary}
\newtheorem{prop}[thm]{Proposition}
\newtheorem{propintro}[thmintro]{Proposition}

\theoremstyle{definition}
\newtheorem{defi}[thm]{Definition}
\newtheorem{defintro}[thmintro]{Definition}
\theoremstyle{remark}
\newtheorem{example}[thm]{Example}

\newtheorem{rem}[thm]{Remark}

\newcommand{\C}{\mathbb C}
\newcommand{\R}{\mathbb R}
\newcommand{\Z}{\mathbb Z}
\newcommand{\Q}{\mathbb Q}
\newcommand{\CP}{\mathbb{CP}}
\newcommand{\RP}{\mathbb{RP}}

\newcommand{\lie}{\mathfrak}

\DeclareMathOperator{\Aut}{Aut}
\DeclareMathOperator{\End}{End}
\DeclareMathOperator{\Span}{span}
\DeclareMathOperator{\GL}{GL}
\DeclareMathOperator{\SU}{SU}
\DeclareMathOperator{\U}{U}

\DeclareMathOperator{\SO}{SO}

\DeclareMathOperator{\G}{G}
\DeclareMathOperator{\D}{D}
\DeclareMathOperator{\E}{E}

\DeclareMathOperator{\Ad}{Ad}
\DeclareMathOperator{\id}{id}
\DeclareMathOperator{\Sym}{Sym}
\renewcommand{\phi}{\varphi}
\renewcommand{\epsilon}{\varepsilon}
\renewcommand{\hat}{\widehat}
\newcommand{\call}{\mathcal}

\DeclareMathOperator{\Spec}{Spec}

\DeclareMathOperator{\res}{res}

\numberwithin{equation}{section}

\title[Laplacian irreducibility of some homogeneous spaces]{Irreducibility of the Laplacian eigenspaces of some homogeneous spaces}
\author{David Petrecca \and Markus R\"oser}
\address{Institut f\"ur Differentialgeometrie, Leibniz Universit\"at Hannover
 \endgraf  Welfengarten 1, 30167
- Hannover - Germany}
\email{petrecca at math.uni-hannover.de, roeser at math.uni-hannover.de}

\begin{document}
\begin{abstract}
	For a compact homogeneous space $G/K$, we study the problem of existence of $G$-invariant Riemannian metrics such that each eigenspace of the Laplacian is a real irreducible representation of $G$. 
	
	We prove that the normal metric of a compact irreducible symmetric space has this property only in rank one. Furthermore, we provide  existence results for such metrics on certain isotropy reducible spaces.
\end{abstract}
\maketitle
\tableofcontents
\section*{Introduction}
Let $(M, g)$ be a compact Riemannian manifold and let $\Delta_g$ be the Laplace operator of $g$ acting on smooth functions.
It is an elliptic, self-adjoint and non-negative operator, so its non-zero eigenvalues $\lambda_i$ are positive and the corresponding  eigenspaces $E_{\lambda_i}$ are finite-dimensional. By definition we have $E_0 = \ker\Delta_g$, which, if $M$ is connected, consists of only the constant functions.

The increasing sequence of real numbers
\[ 
0 = \lambda_0 < \lambda_1 \leq \lambda_2 \leq \ldots \nearrow +\infty,
\]
where each $\lambda_i$ is repeated according to its multiplicity $\dim E_{\lambda_i}$, is called the \emph{spectrum} of $(M,g)$ and denoted by $\Spec(M,g)$.

A metric such that $\dim E_{\lambda_i}=1$, i.e. all eigenspaces are of minimal possible dimension, is called \emph{simple}. If $\call M$ is the space of smooth Riemannian metrics on $M$, Uhlenbeck \cite{Uhlenbeck} proved that the set of simple metrics is residual in $\call M$ with respect to an appropriate topology. We recall that a subset of a topological space is called residual if it is the intersection of countably many sets with dense interior.

Recently, Guillemin, Legendre and Sena-Dias \cite{GuilleminLegendreSenaDias} computed the variation of the Rayleigh quotients within a conformal class, reproving the result of Uhlenbeck stating that every Riemannian metric can be conformally deformed into a metric with simple spectrum.

The notion of simplicity can be generalized in the following way in presence of symmetries. Assume that $(M, g)$ is acted on by a group $G$ of isometries.

The extended action of $G$ on $C^\infty(M,\C)$ given by $(g \cdot f)(x) = f(g^{-1}\cdot x)$ commutes with the Laplacian, so it defines a $G$-representation on each eigenspace $E_\lambda$ which is real in the sense that it preserves the real subspace $C^\infty(M,\R)\cap E_\lambda$. In this context, the spectrum of $\Delta$ will of course not be simple. An analogous minimality property of the eigenspaces is to be irreducible as $G$-representations.

This motivates the following definition.
\begin{defintro}
	A compact Riemannian manifold $(M, g)$ with an isometric action of a group $G$ is said to have a \emph{complex $G$-simple spectrum} if, for every eigenvalue $\lambda\in\Spec(M,g)$, the complex $G$-representation $E_\lambda\subset C^\infty(M, \C)$ is irreducible. If the real $G$-representations $E_\lambda\cap C^\infty(M, \R)$ are all irreducible, we say that $(M,g)$ has \emph{$G$-simple spectrum}, or is \emph{$G$-simple} for short. 
\end{defintro}

It is clear that this property is invariant under homothetic variations of the metric.

This notion first appeared in the monograph \cite{spectre} by Berger, Gauduchon and Mazet. Recently, Schueth \cite{schueth} studied the problem for $M=G$ with a left-invariant metric and $G$ acting by left-translations. She established a number of genericity results for $G$ given by a product of a torus and copies of $\SU(2)$ and quotients of such groups by discrete central subgroups.

In this work, we focus on $G$-invariant metrics on homogeneous spaces $M = G/K$, where $G$ is a compact connected Lie group and $K$ a closed subgroup.  On such spaces, $G$-invariant Riemannian metrics correspond to $\Ad(K)$-invariant inner products on the Lie algebra $\lie g$ of $G$, see e.g. \cite[Chap.~X]{KN}.

Fixing a $\mathrm{Ad}(G)$-invariant inner product on $\lie g$, we get a $K$-invariant decomposition $\lie g = \lie k \oplus \lie m$, where $\lie m = \lie k^\perp$ is called the isotropy representation. 
If $\lie m$ is an irreducible representation of $K$, the space $M$ is called isotropy irreducible, otherwise it is isotropy reducible. The space of left-invariant metrics on $G$ can then be identified with the space $\Sym_K^+(\lie m)$ of $\Ad(K)$-invariant inner products on $\lie m$. If $\lie m$ decomposes into pairwise inequivalent irreducible representations $\lie m_i$, $i=1,\dots,s$, then $\Sym_K^+(\lie m)$ may be identified with the positive orthant in $\R^s$, which we denote by $\R^s_+$. Otherwise, $\Sym_K⁺(\lie m)$ is more complicated since equivalent submodules need not be orthogonal for a $K$-invariant inner product.

The Peter-Weyl theorem describes the decomposition of the $G$-representation $C^\infty(G/K,\C)$ into isotypical components. The representations that occur are exactly the \emph{spherical representations} of the pair $(G,K)$, namely the irreducible complex $G$-representations containing a nonzero subspace of vectors fixed by $K$. The action of the Laplacian on each isotypical component can be described in purely Lie-algebraic terms. 

Adapting Schueth's arguments to our setting, we can give a Lie-algebraic characterization of $G$-simplicity and prove the following result, which means that if a $G$-simple metric on $G/K$ exists, then the generic left-invariant metric on $G/K$ has this property.

\begin{propintro}[see Prop.~\ref{prop:resultants}]
	Let $M=G/K$ be a homogeneous space with isotropy representation $\lie m$. If the space of $G$-simple Riemannian metrics on $G/K$ is non-empty, then it is a residual set in $\Sym_K^+(\lie m)$.
\end{propintro}

Of course, this proposition does not say anything about the problem of existence of $G$-simple metrics on $G/K$. In order to tackle this, one needs to understand the set of spherical representations.

However, in \cite{spectre}, the following result is obtained by analytic techniques, without determining spherical representations.

\begin{thmintro}[Berger-Gauduchon-Mazet \cite{spectre}] \label{thm:BGM}
	Let $G/K$ be a compact homogeneous Riemannian manifold with a $G$-invariant metric and consider the orthogonal splitting $\lie g = \lie k \oplus \lie m$. If the connected component $K^0$ of $K$ acts transitively on the unit sphere of $\lie m$ under the isotropy representation, then $G/K$ has a $G$-simple spectrum.
\end{thmintro}

A homogeneous space $G/K$ satisfying this sufficient condition must be isotropy irreducible. As any two invariant metrics are constant multiples of each other it follows that in this case every invariant metric on $G/K$ is $G$-simple.  

Examples are given by the homogeneous presentations of compact rank-one symmetric spaces (CROSS) along with some non-symmetric spaces described in Examples \ref{ex:G2SU3} and \ref{ex:SO7G2}.

For symmetric pairs, spherical representations are well understood due to the Cartan-Helgason theorem. They are given by \emph{restricted dominant weights} (see Section \ref{sec:restricted}).
In this work, we use this theory to prove that the CROSS's are the only \emph{irreducible} symmetric spaces with this property. 

\begin{thmintro}[see Thm.~\ref{thmrefl}]
	An irreducible symmetric space $G/K$ has simple $G$-spectrum if, and only if, it has rank one.
\end{thmintro}

On the other hand, we obtain reducible $G$-simple symmetric spaces $G/K$ of arbitrary rank by considering Riemannian products of CROSS's.

\begin{thmintro}[see Thm.~\ref{thm:generic1}]
	A generic $G$-invariant metric on the product $G/K = \prod_i G_i/K_i$ of CROSS's is $G$-simple.
\end{thmintro}

Non-symmetric isotropy irreducible homogeneous spaces  have been classified \cite{wolf_class}, but there is in general no efficient parameterization of their spherical representations. It is therefore not clear whether one can adapt the proof of Theorem \ref{thmrefl} or whether the sufficient condition of Theorem \ref{thm:BGM} is necessary.

Full flag manifolds $G/T$, where $T\subset G$ is a maximal torus, form a class of isotropy reducible spaces for which the spherical representations are well understood 
\cite{SpectraFlag}. This is due to the fact that the irreducible spherical representations in this case are exactly those whose highest weight lies in the root lattice. However, the monoid of spherical representations is not freely generated in this case, and this causes difficulties in obtaining sufficient information on the eigenspaces of the Laplacian of a $G$-invariant metric.

If the pair $(G,K)$ is spherical, i.e. for any irreducible representation of $G$ the set of $K$-fixed points is at most one dimensional, then more is known. In fact, Kr\"amer \cite{kraemer} classified spherical pairs $(G,K)$ with $G$ compact, connected and simple and he also gave explicit descriptions of the weights corresponding to spherical representations.

Using these results, we are able to perform explicit computations on a class of spherical spaces that are circle bundles over Hermitian symmetric spaces. Combining well-known properties of the Laplacian on fiber bundles with totally geodesic fibers along with an explicit analysis of a special case, we are able to prove the following theorem for total spaces of Hopf fibrations $S^{2n+1}\to \CP^n$, $n>1$, by stretching the metric along the fibers. The corresponding result for $n=1$ was established by Schueth in \cite{schueth} by a different argument.

\begin{thmintro}[see Prop.~\ref{prop:hopf}]
Let $n>1$. Then a generic $\SU(n+1)$-invariant metric on the total space of the Hopf fibration $S^{2n+1} \to \CP^n$ is $\SU(n+1)$-simple.
\end{thmintro}

We treat in Subsection \ref{sec:SU2F} an example of a quotient of $\SU(2)$ by a non-normal, non-spherical finite subgroup and we prove that a generic $\SU(2)$-invariant metric is $\SU(2)$-simple. This space appeared in \cite{art:bedulli_stab} in the context of homogeneous Lagrangian submanifolds of projective spaces.

The paper is organized as follows. In Section \ref{sec:lapl} we discuss the required background on the Laplacian on a homogeneous space and the Peter-Weyl theorem. We give a purely Lie-algebraic characterization of $G$-simplicity and establish Proposition \ref{prop:resultants}. Then we move to the case of compact symmetric spaces in Section \ref{sec:symm} and prove Theorems \ref{thmrefl} and \ref{thm:generic1}. Finally, we treat the non-symmetric examples in Section \ref{sec:nonsym}.

\subsection*{Acknowledgements}
The authors are supported by the Research Training Group 1463 ``Analysis, Geometry and String Theory'' of the DFG and the first author is supported as well by the GNSAGA of INdAM.  Moreover, they would like to thank Fabio Podest\`a for valuable feedback and his interest in this work and Emilio Lauret for pointing out an inaccuracy in an earlier version of this article.

\section{The eigenspaces of the Laplacian of a homogeneous space}\label{sec:lapl}
\subsection{Preliminaries and notations}
Let $G$ be a compact connected Lie group. Then we have an action of $G\times G$ on $G$ given by 
\[ 
(g,h) \cdot x  = gxh^{-1},\qquad g,h,x\in G.
\]
This action induces an action of $G\times G$ on $C^{\infty}(G,\C)$ via 
\[ 
(g,h) \cdot f(x) = f(g^{-1}xh) =: (L_g\circ R_h)(f)(x).
\]
The maps $g\mapsto L_g, h\mapsto R_h$ define the \emph{left and right regular representations of $G$ on $C^\infty(G,\C)$}.

Denote by $\hat G$ a complete set of representatives for the set of isomorphism classes of irreducible complex representations of $G$. Given $V\in \hat G$, we can define smooth functions on $G$ as follows. Consider the tensor product $V^*\otimes V$ and define for a decomposable element $\phi\otimes v\in V^*\otimes V$ the function $f_{\phi,v}\in C^\infty(G,\C)$ by
\[ 
f_{\phi,v}(x) := \phi(x \cdot v).
\]
In the literature, such a function is called a \emph{matrix coefficient} (see e.g. \cite{BtD}). By linear extension, we obtain a linear map 
\[ 
V^*\otimes V\to C^{\infty}(M,\C).
\]
This map is equivariant with respect to the natural $G\times G$ actions on $V^*\otimes V$ and on $C^\infty(G,\C)$. Since $V$ is irreducible,  the $(G\times G)$-representation $V^*\otimes V$ is irreducible, too. Thus, the map $V^*\otimes V\to C^{\infty}(M,\C)$ must be injective. We therefore have a natural way of realizing $V^*\otimes V$ as a subrepresentation of the $(G\times G)$-representation $C^\infty(G,\C)$.

Let $K\subset G$ be a closed subgroup. Then $G$ acts on $M=G/K$ by left translations via $g \cdot xK = (gx)K$.
Smooth functions on $M$ are given by smooth functions on $G$ which are $K$-invariant with respect to the right-regular action, i.e. functions $f$ such that $f(xk) = f(x)$ for all $x \in G$ and $k \in K$. For $V\in \hat G$ define the subspace of $K$-invariant vectors 
\[ 
V^K = \{v\in V\ : kv = v,\ \forall k\in  K\}.
\]
Recall that a $G$-representation $V$ is \emph{spherical with respect to $(G,K)$} if $\dim V^K > 0$. This property is clearly invariant under the equivalence of representations and we denote by $\hat G_K\subset \hat G$ a complete set of representatives for the equivalence classes of spherical representations with respect to $(G,K)$.
It is clear from the above construction, that 
\[ 
V^*\otimes V^K\subset C^\infty(G,\C)^K = C^\infty(G/K,\C).
\]
Using the Haar measure on $G$, we can consider the Hilbert space completion $L^2(G,\C)$ of $C^\infty(G,\C)$. For $g\in G$, the maps $L_g,R_g:C^\infty(G,\C)\to C^\infty(G,\C)$ extend to bounded linear operators on $L^2(G,\C)$.
The Peter-Weyl theorem states in particular that the space generated by the matrix coefficients, as $V$ ranges over the spherical representations of $(G,K)$, is dense in $L^2(G/K, \C)$.

More precisely, the following holds.

\begin{thm}[Peter-Weyl \cite{sepanski, takeuchi}]
The subspaces
\begin{eqnarray}
\bigoplus_{V \in \hat G} V^*\otimes V &\subset & L^2(G,\C)\\
\bigoplus_{V \in \hat G_K} V^* \otimes V^K &\subset & L^2(G/K,\C)
\end{eqnarray}
are dense.
\end{thm}

For $V\in \hat G$, we denote by $I(V)$ the $V$-isotypical component inside $L^2(G,\C)$ (resp. $L^2(G/K,\C)$) with respect to the left-regular representation. The Peter-Weyl theorem asserts  that $I(V^*) \cong V^*\otimes V\subset L^2(G,\C)$ (resp. $I(V^*)\cong V^*\otimes V^K\subset L^2(G/K,\C)$). Then $\dim V$ (resp. $\dim V^K$) is the multiplicity of $V^*$ in $L^2(G,\C)$ (resp. $L^2(G/K,\C)$).

We say that $V\in \hat G$ is of \emph{real type} (resp. of \emph{quaternionic type}) if there exists a linear $G$-map $J \in  \End(V)$ such that $J^2 = \id$ (resp. $J^2 = -\id$). If $V$ is neither of real nor of quaternionic type, it is called of \emph{complex type}. 

\begin{lemma}[\cite{BtD}]
Let $V$ be an irreducible representation of $G$ and let $V^*$ be its dual representation.
\begin{enumerate}[label=\arabic*.]
\item $V$ is of real type if, and only if, it is self-dual and equals the complexification of an irreducible real $G$-module;

\item $V$ is of quaternionic type if, and only if, it is self-dual and $V \oplus V^*$ equals the complexification of a real irreducible $G$-module;

\item $V$ is of complex type if, and only if, it is not self-dual.
\end{enumerate}
\end{lemma}

\subsection{The Laplacian of a $G$-invariant metric on $M=G/K$}
A $G$-invariant metric on $G/K$ corresponds to an $\Ad(K)$-invariant inner product on the Lie algebra $\lie g$, which in turn corresponds to a $(G\times K)$-invariant metric on $G$.
So let $g$ be a $(G\times K)$-invariant Riemannian metric on $G$. Let $d = \dim \lie g$ and $k = \dim \lie k$. Then $g$ is in particular a left-invariant metric on $G$ and the Laplacian $\Delta_g$ acts on $f\in C^\infty(G,\C)$ by 
\[ 
\Delta_g(f)(x) = -\sum_{i=1}^{d} \frac{d^2}{dt^2} \biggr |_{t=0}f(x\exp(tY_i)),
\]
for an orthonormal basis $\{Y_1,\dots,Y_d\}\subset \lie g$. Let $V\in \hat  G$, $v\in V,\varphi\in V^*$. Then the Laplacian $\Delta_g$ acts on a matrix coefficient $f_{\phi,v}$ by

\[
\Delta_g(f_{\phi,v})(x) = -\sum_{i=1}^{d} \frac{d^2}{dt^2} \biggr |_{t=0}\phi(x\exp(tY_i)v)= -\sum_{i=1}^{d} \phi(x Y_i^2v) = f_{\phi,\Delta_g^Vv}(x),
\]
where
\[ 
\Delta_g^V = -\sum_{i=1}^{d}Y_i^2: V\to V
\]
is the Casimir operator of $V$ with respect to the inner product on $\lie g$ induced by $g$.
Thus, $\Delta_g$ is compatible with the Peter-Weyl decomposition and acts as on a subspace $V^*\otimes V$ by $\id \otimes \Delta_g^V$.
Using the metric we can split $\lie g$ into orthogonal $K$-invariant  vector subspaces $\lie g = \lie k\oplus \lie m$, where $\lie m = \lie k^\perp$. If we choose an orthonormal basis $\{Y_1,\dots,Y_d\}$ adapted to this splitting, i.e. $\{Y_1,\dots, Y_k\}$ is an orthonormal basis for $\lie k$, we get for the action of $\Delta_g$ on $C^\infty(G,\C)^K$ 
\[ 
\Delta_gf(x) = -\sum_{i=k+1}^{d} \frac{d^2}{dt^2}\biggl |_{t=0}f(x\exp(tY_i))
\]
and on the subspace
$V^*\otimes V^K$ the Laplacian $\Delta_g$ acts by $\id\otimes \Delta^{V^K}_g$ where 
\[ 
\Delta_g^{V^K} = \Delta_g^V \biggr |_{V^K} = -\sum_{i=k+1}^{d}Y_i^2: V^K\to V^K.
\]
If we choose a $G$-invariant inner product on $V$, then the Casimir $\Delta^V_g$ is self-adjoint and non-negative, hence diagonalizable with non-negative real eigenvalues.
It follows that the eigenvalues of $\Delta_g$ on the subspace $V^*\otimes V^K$ are given by the eigenvalues of the Casimir $\Delta_g^{V^K}$ on $V^K$. 

Denote by $V_\lambda^K\subset V^K$ the eigenspace of $\Delta_g^{V^K}$ associated with the eigenvalue $\lambda\in\R_{\geq 0}$. Then we see that the left-regular representation of $G$ on $V^*\otimes V^K\subset L^2(G/K,\C)$ decomposes into the eigenspaces of $\Delta_g$ as follows
\[ 
V^* \otimes V^K \cong \bigoplus_\lambda V^*\otimes V^K_\lambda.
\]

In particular, one necessary condition for the eigenspaces $E_\lambda$ to be \emph{complex} irreducible is that $\dim V^K_\lambda = 1$, i.e. $\Delta_g^{V^K}$ should have simple eigenvalues on $V^K$. Another necessary condition is that $\Delta^{V^K}_g$ and $\Delta^{W^K}_g$ do not share any common eigenvalues if $V$ and  $W$ are inequivalent. Note, however, that $(V^*)^K = (V^K)^*$ and that the Casimir of $(V^*)^K$ acts on $\phi\in (V^*)^K$ by 
\[ 
\Delta^{(V^*)^K}_g(\phi)(v) = -\sum_{i} Y_i^2(\phi)(v) =  -\sum_{i} \phi((-Y_i)^2v) = \phi(\Delta_g^{V^K}v).
\]
Thus, the eigenvalues of $\Delta^{(V^*)^K}_g$ and $\Delta^{V^K}_g$ coincide. In particular, $(M,g)$ is not \emph{complex} $G$-simple as soon as there exists a spherical representation $V$ of complex type, i.e. such that $V\ncong V^*$. So for $(M,g)$ to be complex $G$-simple any spherical representation must be of real or quaternionic type. For a spherical representation $V$ of quaternionic type with quaternionic structure $J$ the eigenspaces of $\Delta^V_g$ will be $J$-invariant, which implies that they must be of even dimension. So if $V$ is of quaternionic type, the eigenspaces of $\Delta_g: V^*\otimes V^K\to V^*\otimes V^K$ cannot be irreducible. We summarize our observations in the following proposition.

\begin{prop}\label{Prop:NecCond}
Let $K\subset G$ be closed subgroup of a compact Lie group and consider the homogeneous space $M=G/K$ equipped with a $G$-invariant Riemannian metric $g$. Then the metric is complex $G$-simple, i.e. the eigenspaces $E_\lambda(M,g)\subset L^2(M,\C)$ of the Laplacian $\Delta_g$ are irreducible complex $G$-modules, provided the following hold: 
\begin{enumerate}[label={CI \arabic*}.,ref={CI \arabic*}]
\item \label{CI1} Any spherical representation $V\in \hat G_K$ satisfies $\dim V_\lambda^K=1$ for any eigenspace of the Casimir $\Delta^V_g:  V^K\to V^K$.
\item \label{CI2} Any spherical representation $V\in \hat G_K$ is of real type.
\item \label{CI3} For any pair $V,W\in \hat G_K$ with $V\ncong W$ the Casimir operators $\Delta^{V^K}_g:V^K\to V^K$ and $\Delta^{W^K}_g: W^K\to W^K$ do not have any eigenvalues in common.
\end{enumerate}
\end{prop}

\subsection{Normal homogeneous spaces and spherical pairs}
Suppose the metric $g$ on the homogeneous space $M = G/K$ in the previous paragraph is induced from a bi-invariant metric on $G$. If $M$ is isotropy irreducible, then up to scale this is the only choice of $G$-invariant metric on $M$. 

In this case, for any $V\in \hat G$, the Casimir $\Delta_g^V$ acts as a multiple of the identity on $V$. Hence, $\Delta_g$ acts as a multiple of the identity on $V^*\otimes V^K$. 
Thus, if we look for pairs $(G,K)$ such that the normal homogeneous space $G/K$ has complex $G$-simple spectrum, then by Proposition \ref{Prop:NecCond} we necessarily need $\dim V^K= 1$ for any spherical representation. This motivates the following definition from \cite{kraemer}.

\begin{defi}
Let $K$ be a subgroup of $G$. The subgroup $K$ is \emph{spherical} if every spherical representation $V\in\hat G_K$ satisfies $\dim V^K = 1$. In this case we call $(G,K)$ a \emph{spherical pair}.
\end{defi}
 
When  $G$ is compact, connected and simple, Kr\"amer  classified all possible connected spherical subgroups $K \subset G$ and the list of the weights generating the subset of spherical representations (the \emph{spherical weights}) is given in the fifth column of his table \cite[p.~149]{kraemer}.

The spectrum of a normal homogeneous space can be calculated explicitly using Freudenthal's formula for the eigenvalues of the Casimir operator. If $V_\rho\in\hat G_K$ is the spherical irreducible representation with dominant highest weight $\rho$, then the Laplacian $\Delta^V_g$ acts on $V_\rho^*\otimes V_\rho^K$ with eigenvalue 
\[ 
\lambda(\rho) = (\rho+2\delta,\rho)
\]
where $(\cdot , \cdot)$ denotes the inner product on $\lie g$ induced by the bi-invariant metric and $\delta$ is the sum of the fundamental weights \cite{takeuchi}. An explicit analysis of the function $\rho\mapsto \lambda(\rho)$ will allow us later to determine the $G$-simple compact irreducible symmetric spaces.

\subsection{Real eigenspaces}
The property to require all eigenspaces of the Laplacian $\Delta_g$ associated with a $G$-invariant metric inside $L^2(G/K,\C)$ to be complex irreducible representations fails as soon as there is a quaternionic spherical representation of $G$. In this paragraph we will consider the eigenspaces of $\Delta_g$ inside $L^2(G/K,\R)$. 

In what follows, we will adapt Schueth's \cite{schueth} line of argument to the homogeneous setting. 

We first observe that for any $V\in \hat G$ a choice of invariant hermitian metric induces an isomorphism $\bar V\cong V^*$. For $V\in \hat G_K$ we denote the subspace $V^*\otimes V^K\subset L^2(G/K,\C)$ by $I(V^*)$, as it is the isotypical component associated with the isomorphism class of $V^*$ with respect to the left-regular representation. Define 
\[ \mathcal C_V = I(V^*)+I(V) = \begin{cases} I(V^*) & \text{if $V$ is of real or quaternionic type,}\\
 I(V^*)\oplus I(V) &\text{if $V$ is of complex type.}\end{cases}\]
On $C^\infty(G/K,\C)$ we have a natural real structure given by complex conjugation $f\mapsto \bar f$, which induces a map $I(V^*) = I(\bar V)\to I(V)$ and thus preserves $\mathcal C_V$.
We further put 
\[\mathcal E_V = \mathcal C_V\cap C^\infty(G/K,\R),\]
which is invariant under the left-regular action of $G$.
Then it follows that $\mathcal C_V = \mathcal E_V\otimes \C$. The Laplacian is a real operator and hence acts on $\mathcal E_V$. The following Lemma gives criteria for the eigenspaces of $\Delta_g|_{\mathcal E_V}$ to be irreducible. 

\begin{lemma}
Let $\mu$ be an eigenvalue of $\Delta^{V^K}_g|_{\mathcal E_V}$ with multiplicity $m$. Then the  corresponding eigenspace is irreducible as a real $G$-representation if and only if $m$ is minimal, i.e.
\[ 
m = \begin{cases}1 & \text{if $V$ is of real or complex type}\\ 2 & \text{if $V$ is of quaternionic type.} \end{cases}
\]
\end{lemma}
\begin{proof}
Let $\mu$ be an eigenvalue of $\Delta^{V^K}_g$ with multiplicity $m$, which is always even if $V$ is of quaternionic type. Let $V^K_\mu\subset V^K, (V^*)^K_\mu\subset V^*, \mathcal C^\mu_V, \mathcal E^\mu_V$ be the corresponding eigenspaces of $\Delta^{V^K}_g,\Delta_g^{(V^*)^K}$. 

As a complex $G$-representation $\mathcal C^\mu_V$ satisfies 
\[ \mathcal C^\mu_V = \begin{cases} V^*\otimes V^K_\mu \cong (V^*)^{\oplus m}\cong V^{\oplus m} & \text{if $V$ is of real or quaternionic type}\\ V^*\otimes V^K_\mu \oplus V\otimes (V^*)^K_\mu \cong (V^*\oplus V)^{\oplus m} &\text{if $V$ is of complex type.}\end{cases} \]
This implies that there exists in each case an irreducible real $G$-representation $U$ such that 
\[ \mathcal C^\mu_V = \begin{cases} U^{\oplus m}\otimes \C & \text{if $V$ is of real type}\\ U^{\oplus m}\otimes \C  &\text{if $V$ is of complex type} \\
U^{\oplus m/2}\otimes \C & \text{if $V$ is of quaternionic type.}\end{cases}\]
Now $C^\mu_V$ is the complexification of $\mathcal E^\mu_V$, i.e. as a real $G$-representation we have 
\[ \mathcal C^\mu_V\cong \mathcal E^\mu_V\oplus i\mathcal E^\mu_V \cong \mathcal (\mathcal E^\mu_V)^{\oplus 2}.\]
Thus the decomposition of $\mathcal E^\mu_V$ into irreducible real representations is 
\[ 
\mathcal E^\mu_V = \begin{cases} U^{\oplus m}& \text{if $V$ is real}\\ U^{\oplus m}  &\text{if $V$ is of complex type} \\
U^{\oplus m/2} & \text{if $V$ is quaternionic.}\end{cases}
\]
We see that this is irreducible if $m=1$ in the real and complex cases and $m=2$ in the quaternionic case.
\end{proof}

\begin{prop}\label{IrrCond}
Let $g$ be a $G$-invariant metric on $M=G/K$. Then the eigenspaces of the Laplacian $\Delta_g$ inside $L^2(M,\R)$ are irreducible $G$-modules if, and only if, the following three conditions are all satisfied: 
\begin{enumerate}[label={RI \arabic*.},ref={RI \arabic*}]
\item \label{irrcondA} For any pair $V,W\in \hat G_K$ with $V\ncong W$ and $V^*\ncong W$ the Casimir operators $\Delta^{V^K}_g,\Delta^{W^K}_g$ have no common eigenvalues. 
\item \label{irrcondB} For any $V\in \hat G_K$ of real or complex type, the eigenvalues of $\Delta^{V^K}_g$ are simple. 
\item \label{irrcondC} For any $V\in \hat G_K$ of quaternionic type, the eigenvalues of $\Delta^{V^K}_g$ are of multiplicity exactly two.
\end{enumerate}
\end{prop}

\begin{corol}
If $(G,K)$ is a spherical pair, then the Laplacian associated with a $G$-invariant metric on $G/K$ has irreducible real eigenspaces if it has irreducible complex eigenspaces. Conversely, if every spherical representation is of real type then real irreducibility implies complex irreducibility.
\end{corol}

\begin{proof}
If $(G,K)$ is a spherical pair, then conditions \eqref{irrcondB}, \eqref{irrcondC} and \eqref{CI1} are always satisfied. Let $(G,K)$ have irreducible complex eigenspaces. Then condition \eqref{CI3} implies condition \eqref{irrcondA}, so $(G,K)$ has also real irreducible eigenspaces.

Conversely, if $(G,K)$ has real irreducible eigenspaces and every spherical representation is of real type, then condition \eqref{CI2} is satisfied by assumption and condition \eqref{irrcondA} implies condition \eqref{CI3}.
\end{proof}

Suppose we are given $M=G/K$ and a decomposition 
\[
\lie g = \lie k \oplus \lie m
\]
as above, orthogonal with respect to a fixed $\mathrm{Ad}(G)$-invariant inner product $\langle \cdot, \cdot \rangle$ on $\lie g$. 
Any $\Ad(K)$-invariant inner product $g \in \Sym_K^+(\lie m)$ on $\lie m$ can be written as $g = g_B$, where $g_B:= \langle B^{-1} \cdot, \cdot \rangle$ for a $K$-invariant symmetric positive-definite endomorphism $B \in \Aut_K(\lie m)$.

If $\{ Y_i\}$ is an orthonormal basis of $\lie m$ with respect to $\langle \cdot, \cdot \rangle$, then we can find $A\in \GL(\lie m)$ such that $\{ A Y_i\}$ is an orthonormal basis with respect to $g_B$.  This implies that $B = AA^T$, where the transpose is taken with respect to $\langle\cdot,\cdot\rangle$. We write $AY_i = \sum_j A_{ji}Y_j$. 

For $V\in\hat G_K$, we consider the associated Casimir operator $\Delta^V_B:= \Delta_{g_B}^{V^K}:V^K\to V^K$ given by
\[
\Delta^V_B  = \sum_{i=1}^n (AY_i)^2= \sum_{i,j=1}^n B_{ij}Y_iY_j.
\]
To each $V\in\hat G_K$ we associate the map
\[
p_V:\Sym_K^+(\lie m)\to \C[t],\quad p_V(B)(t) = \det (t \cdot \id_{V^K}-\Delta^V_B)\in\C[t],
\]
i.e. $p_V(B)$ is the characteristic polynomial of $\Delta^V_B$. 

Using the resultant
\[ 
\res: \C[t]\times \C[t]\to \C
\]
which is a polynomial in the coefficients of $p$ and $q$ with the property that 
\[ 
\res(p,q) =0 \qquad \text{$\iff$ $p,q\in\C[t]$ have a common root},
\]
we can rephrase the conditions in Proposition \ref{IrrCond} as follows. Let $V,W\in \hat G_K$. Then $\Delta_B^V$ and $\Delta^W_B$ have no common eigenvalues if and only if $\res(p_V(B),p_W(B)) \neq 0$. We denote by $p_V',p_V''$ the first and second derivative (with respect to $t$) of $p_V$. Then the polynomial $p_V(B)\in \C[t]$ has only simple roots if, and only if, $\res(p_V(B),p_V'(B)) \neq 0$. Thus, there exists $B \in \Sym_K^+(\lie m)$ such that $\Delta^V_B$ has only simple eigenvalues (resp. only eigenvalues of multiplicity at most two) iff the polynomial map $\res\circ(p_V,p_V'):\Sym_K^+(\lie m) \to \C,\quad B \mapsto \res(p_V(B),p_V'(B))$ (resp. $\res\circ(p_V,p_V''):\Sym_K^+(\lie m)\to \C$)  is not the zero map. 

Notice that the formula for $\Delta_B^V$ and hence the polynomial $p_V(B)\in \C[t]$ makes sense for any $B \in \Sym_K(\lie m)$, although the metric $g_B$ of course does not. Moreover, the polynomial in $B$ given by $\res(p_V(B),p_V'(B))$  vanishes on all of $\Sym_K(\lie m)$ if and only if it vanishes on the open subset $\Sym_K^+(\lie m)$.

This allows us to make the following observation, which is analogous to Proposition 3.7 in \cite{schueth}.

\begin{prop} \label{prop:resultants}
The existence of a $G$-invariant metric on $M=G/K$ such that the associated Laplacian has irreducible real eigenspaces is equivalent to the following three conditions being simultaneously satisfied:

\begin{enumerate}[label={\alph*.},ref={\alph*}]
\item \label{irrcondresA} For all $V,W\in \hat G_K$ such that $V\ncong W$ and $V^*\ncong W$ the polynomial 
\[
\res \circ (p_V,p_W)\colon \Sym_K(\lie m) \to \C
\]
is not the zero polynomial.

\item \label{irrcondresB} For all $V\in\hat G_K$ of real or complex type the polynomial 
\[
\res \circ (p_V,p_V')\colon \Sym_K(\lie m) \to \C
\] is not the zero polynomial.  

\item \label{irrcondresC} For all $V\in\hat G_K$ of real or quaternionic type the polynomial
\[
\res \circ (p_V,p_V'') \colon \Sym_K(\lie m)\to\C
\] is not the zero polynomial. 
\end{enumerate}

In particular, the space of $G$-simple invariant metrics on $G/K$ is a residual set in $\Sym^+_K(\lie m)$, provided that it is non-empty.
\end{prop}
\begin{proof}
It is obvious that the conditions are just reformulations of the corresponding conditions in Proposition \ref{IrrCond}. It follows that the complement of the space of $G$-simple metrics inside $\Sym_K^+(\lie m)$ is given as the intersection of a countable union of vanishing loci of non-trivial polynomials on $\Sym_K(\lie m)$ with the open subset $\Sym_K^+(\lie m)$. Hence the set of $G$-simple metrics on $G/K$ is a residual set, provided it is not empty.
\end{proof}

\begin{rem}
If the isotropy representation $\lie m$ decomposes into a direct sum of $s>0$ inequivalent irreducible representations $\lie m_i$, $i=1,\dots,s$, 
\[
\lie m = \lie m_1\oplus\dots\oplus \lie m_s,
\] 
then one can obtain a simple description of $\Sym_K^+(\lie m)$. For any $K$-invariant inner product the decomposition must be orthogonal and up to scale each $\lie m_i$ carries a unique $K$-invariant inner product by irreducibility. In other words, any metric is of the form $g = \sum_{i=1}^s \beta_i^{-1}\langle\cdot,\cdot\rangle|_{\lie m_i}$ for suitable coefficients $\beta_i>0$ and we call such metrics of diagonal type. If we write such a metric in the form $g_B$ as above, then the matrix $B$ must be block diagonal with respect to the above decomposition, where the blocks are just given by $\beta_i\mathrm{id}_{\lie m_i}$. Hence in this case we may identify $\Sym_K^+(\lie m)$ with the positive orthant $\R^s_+\subset \R^s$. 

If two or more of the $\lie m_i$'s are equivalent as $K$-representations, then their sum need no longer be orthogonal with respect to a $K$-invariant inner product, which therefore need not be of diagonal type. This means that the possible matrices $B$ need no longer be block diagonal. Using that the $\R$-algebra $\mathrm{End}_K(\lie m_i)$ is a division algebra by Schur's lemma and so isomorphic to $\R,\C$ or $\mathbb H$, where $\mathbb H$ denotes the quaternions, one can then in principle describe the possible ``off-diagonal'' blocks of such $B$ explicitly.
\end{rem}

We conclude this  section with a generalization of Lemma~4.9 of \cite{schueth}.
\begin{lemma} \label{lemma:covering}
Let $G/K$ be a homogeneous space satisfying all the conditions in Proposition \ref{IrrCond} with respect to some $G$-invariant metric $g$. Let $\Gamma \subset G$ be a discrete central subgroup and let $\bar G = G/\Gamma$ and $\bar K = K / (K \cap \Gamma)$. Then the metric $g$ descends to a $\bar G$-invariant metric on $\bar G / \bar K$ satisfying the conditions of Proposition \ref{IrrCond}.
\end{lemma}

\begin{proof}
We can identify $\bar G / \bar K$ with $(G/K)/\Gamma$, where the action of $\Gamma$ is by translations. It is then clear that, since $\Gamma$ is discrete and central, every $G$-invariant metric on $G/K$ descends to a $\bar G$-invariant metric on $\bar G / \bar K$. We thus have a bijective correspondence between $G$-invariant metrics on $G/K$ and $\bar G$-invariant metrics on $\bar G/\bar K$. On the Lie algebra level we have $\bar{\lie g} = \lie g, \bar{\lie k} = \lie k , \bar{\lie m} = \lie m$. This gives an identification of $\Sym_K^+(\lie m)$ with the $\bar G$-invariant metric on $\bar G/\bar K$

We observe that the spherical representations of the pair $(\bar G, \bar K)$ are given by the spherical representations of $(G,K)$ on which $\Gamma$ acts trivially.  Such representations are of real, complex or quaternionic type as $\bar G$-representation if and only if they are of respective type as $G$-representations.

So if $(G,K)$ satisfies the conditions \eqref{irrcondresA}, \eqref{irrcondresB}, and \eqref{irrcondresC} of Proposition \ref{prop:resultants}, then they are also satisfied for $(\bar G, \bar K)$, since for the latter space we only have to check the conditions on the subset of representations on which $\Gamma$ acts trivially.
\end{proof}

\section{Compact symmetric spaces}\label{sec:symm}

In this section we apply the above results to compact irreducible symmetric spaces $G/K$. It is known that such symmetric pairs $(G,K)$ are spherical, see e.g. \cite{kraemer, takeuchi}. Our aim in this section is to show that the only irreducible symmetric spaces whose associated Laplacian has irreducible real or complex eigenspaces are the CROSS's.

\subsection{Spherical representations of a compact symmetric pair}\label{sec:restricted}
At the Lie algebra level, let $(\lie g, \lie k)$ be a compact symmetric pair with Cartan splitting $\lie g = \lie k \oplus \lie p$. Recall that its rank is defined to be the dimension $\ell$ of a maximal Abelian subalgebra of $\lie p$. Denote by $r$ the rank of $\lie g$.

Let $\lie h$ be a Cartan subalgebra of the complexification $\lie g^\C$.
It can be split as $\lie h = \lie a \oplus \lie a_{\lie k}$ and note that $\lie a$ and $\lie a_{\lie k}$ are orthogonal with respect to the Killing form.

Let $\alpha$ be a root of $(\lie g^\C, \lie h)$. We denote by $\tilde \alpha$ its restriction to $\lie a$. If $\beta \in \lie a^*$ we consider the space
\begin{equation} \label{resrootspace}
\tilde{\lie g}_{\beta} = \{ X \in \lie g^\C: [A, X] = \beta(A) X \text{ for all $A \in \lie a$} \}.
\end{equation}

A \emph{restricted root} is a functional $\beta \in \lie a^*$ such that $\tilde{\lie g}_{\beta} \neq 0$ and the \emph{restricted root system} of $(\lie g, \lie k)$ is defined to be
\[ 
\Sigma(\lie g, \lie k) := \{ \beta \in \lie a^*: \beta \neq 0, \tilde{\lie g}_{\beta} \neq 0 \}.
\]
Given $\alpha\in \lie h^*$, we denote by $\tilde \alpha\in \lie a^*$ its restriction to $\lie a$.
From the root system $\Sigma(\lie g)$ of $(\lie g^\C, \lie h)$, the restricted root system can be recovered as
\[ 
\Sigma(\lie g, \lie k) = \{ \tilde \alpha: \alpha \in \Sigma(\lie g), \alpha|_{\lie a} \neq 0 \}.
\]

The set of restricted simple roots is

\[ 
\Pi (\lie g, \lie k) = \{ \tilde \alpha: \alpha \in \Pi(\lie g), \alpha|_{\lie a} \neq 0 \},
\]
where $\Pi(\lie g)$ is a system of simple roots for $\lie g$.

The set of restricted roots $\Sigma(\lie g, \lie k)$ forms a \emph{restricted root system} in $\lie a^*$, namely they span $\lie a^*$ and they are closed under reflections.

Unlike root systems in the classical sense, if $\alpha$ is in a restricted root system $\Sigma$, then also its multiples $\pm 2 \alpha$ and $\pm \frac 1 2 \alpha$ may occur in $\Sigma$. Moreover, the dimension of the restricted root space \eqref{resrootspace} can be greater than one and this dimension is called the \emph{multiplicity} of the restricted root $\beta$ and denoted by $m_\beta$. Equivalently $m_\beta$ is given by the number of roots $\alpha\in \Sigma(\lie g)$ which satisfy $\tilde\alpha = \beta$. 

Let $\{ \alpha_1, \ldots, \alpha_r \}$ be simple roots of $\lie g$ and let $\{\gamma_1, \ldots, \gamma_\ell \}$ be simple restricted roots of  $(\lie g, \lie k)$.

As in Takeuchi \cite{takeuchi}, let 
\[ 
\beta_i = \begin{cases} 2 \gamma_i & \text{ if $2\gamma_i$ is a restricted root}\\ \gamma_i & \text{ otherwise}. \end{cases}
\]  

A \emph{restricted weight} is a functional $\lambda \in \lie a^*$ such that $\frac{2(\lambda, \alpha)}{(\alpha, \alpha)} \in \Z$ for all restricted roots $\alpha$.

The \emph{fundamental restricted weights} $M_i$ satisfy
\begin{equation} \label{fundwg}
\frac{(M_i, \beta_j)}{(\beta_j, \beta_j)} = \delta_{ij}.
\end{equation}

For the theory of restricted root systems we refer to Helgason \cite{helgason}. 

The spherical weights, i.e. the highest weights of spherical representations, of $(G,K)$ can be characterized in the following way.

\begin{thm}[Cartan-Helgason]
Let $(G,K)$ be a compact symmetric pair and let $\lambda$ be the highest weight of a representation $V$ of $G$. Then $V$ if spherical with respect to $K$ if, and only if,
\begin{equation} \label{eq:CartanHelgason} 
\lambda|_{\lie a_{\lie k}} = 0 \text{ and } \frac{2 (\lambda, \alpha)}{(\alpha, \alpha)} \in \Z_+ \text{ for all positive $\alpha \in \Sigma(G,K)$}.
\end{equation}
\end{thm}

For this result, we refer to \cite{helgason1993geometric}, see also \cite{takeuchi}.

This means that spherical representations of $(G,K)$ are parameterized by dominant restricted weights, which can be written uniquely as $\lambda = \sum_k a_k M_k$ with the $M_k$ defined in \eqref{fundwg} and coefficients $a_k\in\Z_+$.

The data given by the fundamental root system of $(\lie g, \lie k)$ can be represented graphically by means of the \emph{Satake diagram}, see e.g. \cite{SatakeDiag}.

One starts with the Dynkin diagram of $\lie g$ and paints the node corresponding to the simple root $\alpha$ in black if it vanishes on ${\lie a}$ and in white if it does not. Moreover, two simple roots $\alpha$ and $\beta$ are joined by a double tipped arrow if they are related by the \emph{Satake involution}, see \cite{takeuchi}.

If $\pi_i$ are the fundamental weights of $\lie g$, then the fundamental restricted weights are given by
\begin{equation}
M_i = \begin{cases}
\pi_i \text{ if $\alpha_i$ has no arrows and is connected to a black node}\\
2 \pi_i \text{ if $\alpha_i$ has no arrows and is  not connected to a black node}\\
\pi_i + \pi_j \text{ if $\alpha_i$ and $\alpha_j$ are joined by an arrow}.
\end{cases}
\end{equation}

We isolate the following Lemma for future reference. It can be proven by examining the list of restricted root systems and Satake diagrams for simple Lie algebras from \cite{SatakeDiag, Sugiura}.

\begin{lemma} \label{lemma:selfdual}
Let $(G,K)$ be symmetric pair and let $\lambda = \sum_k m_k M_k$ be the highest weight of a spherical representation, where the $M_k$ are fundamental restricted weights. Then the dual representation is given by the highest weight $\lambda^* = \sum_k m_{\sigma k} M_k$, where $\sigma$ is the permutation corresponding to the canonical involution of the restricted root system.
\end{lemma} 

\subsection{The weight $\delta$} \label{sec:delta}

Let $\delta \in \lie h^*$ be the strictly dominant weight given by the half sum of all positive roots of $\lie g$. Its components along $\lie a_{\lie k}^*$ and $\lie a^*$ are given by Lemma~5.1 of \cite{GinGoo}.

\begin{lemma}
The component along $\lie a^*$ of $\delta$ is given by
\begin{equation} \label{eq:bardelta}
\bar \delta = \frac 1 2 \sum_{\xi \in \Sigma^+(\lie g, \lie h)} m_\xi \xi,
\end{equation}	
where $m_\xi$ is the multiplicity of the positive restricted root $\xi$.
\end{lemma}	

We want to compute the coefficients of $\bar \delta$ with respect to the basis of restricted fundamental weights $\{M_i\}$. We have the following Lemma. 

\begin{lemma}
Let $\bar \delta = \sum_{i=1}^\ell k_i M_i$ be the expression of $\bar \delta$ with respect to the basis of fundamental restricted weights. Then we have
\begin{equation}
k_i = 
\begin{cases}
\frac 1 2 m_{\gamma_i} & \text{ if $m_{2 \gamma_i} = 0$} \\
\frac 1 4 ( m_{\gamma_i} + 2 m_{2 \gamma_i}) & \text{ if $m_{2 \gamma_i} \neq 0$}.
\end{cases}
\end{equation}
\end{lemma}

\begin{proof}
The proof goes along a similar line as the classical case, but taking into account that we have a restricted root system.

Consider the transformation $\sigma_i$ in the restricted Weyl group associated to the simple restricted root $\gamma_i$. It is proved in Helgason \cite[Lemma~2.21, Chap.~VII]{helgason} that $\sigma_i$ permutes all the positive roots not proportional to $\gamma_i$ and that, in our base, such roots are $\gamma_i$ and, in case, $2 \gamma_i$.

Taking $\bar \delta$ as in \eqref{eq:bardelta}, arguing as in the classical case, we have that
\[
\sigma_i \bar \delta = \bar \delta - (m_{\gamma_i}+ 2 m_{2 \gamma_i}) \gamma_i.
\]

We then have
\begin{align*}
( \bar \delta - (m_{\gamma_i}+ 2 m_{2 \gamma_i}) \gamma_i, \gamma_i)	&= (\sigma_i^2 \bar \delta, \sigma_i \gamma_i) \\
	&= (\bar \delta, -\gamma_i).
\end{align*}

So we have that
\[
- \frac{(\bar \delta, \gamma_i)}{(\gamma_i, \gamma_i)} =  \frac{(\bar \delta, \gamma_i)}{(\gamma_i, \gamma_i)} - (m_{\gamma_i}+ 2 m_{2 \gamma_i}),
\]
that means
\begin{equation} \label{eq:prodgamma}
\frac{(\bar \delta, \gamma_i)}{(\gamma_i, \gamma_i)} = \frac 1 2 (m_{\gamma_i}+ 2 m_{2 \gamma_i}).
\end{equation}
Recalling how $\beta_i$ is defined, together with the relation \eqref{fundwg}, the expression \eqref{eq:prodgamma} gives us the claim.
\end{proof}
	
We summarize in Table \ref{tab:delta} the coefficients of $2 \bar \delta$ for every Cartan type.
	
\begin{table}[htbp]
\newcolumntype{C}{>{$}l<{$}}
\begin{center}
\begin{tabular}{cCC}
	Cartan type				& 2 \bar \delta 					& \text{Range}\\
	\toprule
	A I						& (1,\ldots, 1)						& \\
	A II						& (4, \ldots, 4)				&\\
	A III (1)				& (2, \ldots, 2, r -2 \ell+ 2)			& r \geq 4 \\
	A III (2)				& (2, \ldots, 2,1)					&\\
	B I						& (1, \ldots, 1, 2(r-\ell)+1) 		& r \geq 2\\
	C I						& (1, \ldots, 1)					&\\		
	C II	(1)					& (4, \ldots, 4, 2(r-2\ell) +3) & r \geq 3\\
	C II (2)					& (4, \ldots, 4,3)				&\\
	D I (1)					& (1, \ldots, 1,2)					&\\
	D I (2)					& (1, \ldots, 1, 2(r-\ell))			& r \geq 4\\
	D I (3)					& (1,\ldots, 1)						&\\
	D III (1)				& (4, \ldots, 4,1)					&\\
	D III (2)				& (4, \ldots, 4,3)					&\\
\midrule
	E I						& (1,1,1,1,1,1)						&\\
	E II					& (2,1,2,1)							&\\
	E III					& (5,6)								&\\
	E IV					& (8,8)								&\\
	E V						& (1,1,1,1,1,1,1)					&\\
	E VI					& (1,1,4,4)							&\\
	E VII					& (8,8,1)							&\\
	E VIII					& (1,1,1,1,1,1,1,1)					&\\
	E IX					& (1,1,8,8)							&\\
	F I						& (1,1,1,1)							&\\
	G						& (1,1)								&\\
\bottomrule
\end{tabular}
\end{center}
\caption{Weight $2\bar \delta$ for every Cartan type of symmetric spaces of rank $\ell \geq 2$. The integer $r$ and its range are like in Helgason's table \cite[Table~VI]{helgason}. The numbers in round brackets refer to the cases as ordered in the same table.} \label{tab:delta}
\end{table}  

\subsection{Rank-two symmetric spaces}
In Table \ref{tab:rank2} we isolate, for future reference, the rank two symmetric spaces whose weight $2 \bar \delta$ has different coefficients.
\begin{table}[htbp]
	\newcolumntype{L}{>{$}l<{$}}
	\newcolumntype{C}{>{$}c<{$}}
	\begin{center}
		\begin{tabular}{cLLC}
								Cartan type	& 2 \bar \delta &	\text{Range}&\text{Restricted type}\\
			\toprule
{A III (1)}					& (2, r - 2) 	&r \geq 4		&C_2\\
{A III (2)}			& (2,1)			& 				&C_2	\\
{B I}				& (1, 2r-3) 	&r\geq 2		&B_2\\
{C II (1)}				& (4, 2r-5) 	&r\geq 5		&B_2\\
{C II (2)}				& (4,3)			& 				&B_2\\
{D I (1)}				& (1,2)			& 				&B_2\\
{D I (2)}				& (1, 2(r-2))	&r \geq 4		&B_2\\
{D III (1)}			& (4,1)			& 				&B_2\\
{D III (2)}			& (4,3)			& 				&C_2\\
			\midrule
{E III}				& (5,6)			& 				&B_2\\
\bottomrule
\end{tabular}
\end{center}
\caption{Rank two symmetric spaces for which $2 \bar \delta$ is not proportional to $M_1 + M_2$.} \label{tab:rank2}
\end{table}  

Using relation \eqref{fundwg}, it follows that the Gram matrix of the Killing form on $\lie a^*$ with respect to the basis $\{M_1, M_2\}$ is given, up to a multiple, by
\begin{equation} \label{eq:gram}
g = 2 \bar m^{-T} \begin{pmatrix}(\beta_1, \beta_1)&0\\0& (\beta_2, \beta_2) \end{pmatrix},
\end{equation} 
where $\bar m$ denotes the Cartan matrix of the restricted root system.

For the cases described by Table \ref{tab:rank2}, we have
\begin{equation} \label{eq:gram2}
g = \begin{cases}
\begin{pmatrix}
2&2\\2&4
\end{pmatrix} & \text{ if the type is $B_2$}\\
& \\
\begin{pmatrix}
4&2\\2&2
\end{pmatrix} & \text{ if the type is $C_2$.}\\
\end{cases}
\end{equation}

\subsection{The main theorem}
We are now in a position to state and prove one of our main results.
\begin{thm} \label{thmrefl}
Let $(G, K)$ be a compact irreducible symmetric pair with simple $G$-spectrum (in either the complex or real sense). Then $(G,K)$ has rank one.
\end{thm}

We split the proof in the two following Proposition and an analysis of the rank-two case.

\begin{prop}
Let $(G, K)$ be a compact irreducible symmetric pair of rank $\ell \geq 3$. Then $G/K$ is not $G$-simple.
\end{prop}
\begin{proof}
Our aim is to show that a pair $(G,K)$ of rank greater than one violates condition \eqref{irrcondA} in Proposition \ref{IrrCond}.
Let $\lie g = \lie k \oplus \lie m$ be a Cartan splitting of $\lie g$, with maximal Abelian subalgebra $\lie a \subseteq \lie m$ and denote by  $\ell\geq 2$ the rank of $(G,K)$.
Let  $\{\beta_1, \ldots, \beta_\ell \}$ be simple restricted roots, with associated fundamental weights $\{ M_1, \ldots, M_\ell \}$.

The irreducible spherical representations are parameterized by the cone of dominant restricted weights. Each such weight is of the form 
\[ 
\rho = \sum_{i=1}^\ell n_iM_i,
\]
where $n_i$ are non-negative integers. By Freudenthal's formula the representation $V_\rho$ is contained in the eigenspace of $\Delta$ with eigenvalue 
\[ 
\lambda_\rho = (\rho+2\delta,\rho).
\]
Thus, the symmetric space $G/K$ will have $G$-simple spectrum if the function $\rho\mapsto (\rho+2\delta,\rho)$ is injective on the cone of dominant restricted weights. Note that if $R:\lie a^*\to\lie a^*$ is an isometry fixing $\delta$, then $(R\rho+2\delta,R\rho) = (\rho+2\delta,\rho)$. 

Thus, in order to show that $G/K$ does not have $G$-simple spectrum we will construct an isometry $R$ and a dominant restricted weight $v$ such that $R\delta = \delta$ and $Rv$ is again a dominant restricted weight. 

Our strategy is to construct, for rank $\ell \geq 3$, a suitable restricted weight $\alpha$ which is perpendicular to $\delta$ and to take the reflection $R_\alpha: v\mapsto  v-\frac{2(\alpha,v)}{(\alpha,\alpha)}\alpha$. This clearly fixes $\delta$.

Let $(c_{ij})$ be the Gram matrix of the basis $\{\beta_1, \ldots, \beta_\ell \}$. Note that since for any $i,j$ the number $\frac{2(\beta_i,\beta_j)}{(\beta_i,\beta_i)}$ is an integer and also the square norms $(\beta_i,\beta_i)$ are rational, the numbers $c_{ij}$ must also be \emph{rational}.

The components of the projection on $\lie a^*$ of the dominant weight $2\delta$, that we call still $2\delta$ in this proof, of $\lie g$ were computed in Table \ref{tab:delta}. From such table we can identify, for all irreducible symmetric spaces, a pair of fundamental restricted roots that we relabel $\beta_1$ and $\beta_2$, such that $2 \delta = M_1 + M_2 + \sum_i k_i M_i$. From that computation, together with \cite[\S.~5, Lemma~4]{takeuchi}, we also see that $\beta_1$ and $\beta_2$ have the same square norm, say, $2$. 

Consider $\alpha = \beta_1 - \beta_2\in\lie a^*$. By the above remark, we have $(\alpha, \delta)=0$. 

Of course, this relation only depends on the direction defined by $\alpha$, i.e. we may multiply it by a non-zero real number. In the basis $\{M_i\}$ we can write 
\[ \alpha = \sum_{i=1}^\ell \frac{(\alpha,\beta_i)}{(\beta_i,\beta_i)}M_i.\]
Since 
\[ (\alpha,\beta_i) = c_{1i}-c_{2i},\] 
we may thus rescale $\alpha$ by some positive integer in such a way that $(\alpha,\beta_i)$ and $\frac{(\alpha,\beta_i)}{(\beta_i,\beta_i)}$ are also integers for any $i=1,\dots,\ell$.

The square length of $\alpha$ is
\begin{equation}
(\alpha, \alpha) = 4 - 2 c_{12},
\end{equation}
which is a positive integer. 

We notice that, by construction, $R_\alpha$ fixes not only $\delta$ but also the fundamental weights $M_3, \ldots, M_\ell$. Our goal is to find some dominant $v$ such that $R_\alpha(v)$ is dominant. 

We make an ansatz for $v$ to be the form $v= M_1 + \sum_{k=3}^\ell m_kM_k$ with non-negative integers $m_k$ to be chosen later. We have 
\[ R_\alpha v = R_\alpha M_1 + \sum_{k=3}^\ell m_kM_k = \sum_{k=1}^\ell \frac{(R_\alpha v,\beta_k)}{(\beta_k,\beta_k)}M_k\]
and for this to be a dominant restricted weight, we need the coefficients to be
\[ \frac{(R_\alpha v,\beta_k)}{(\beta_k,\beta_k)} \geq 0, \text{ for all $k$.}\]

Note that the component of $R_\alpha v$ along $M_1$ and $M_2$ is the same as the component of $R_\alpha M_1 $ along these vectors.
The component of $R_\alpha M_1$ along $M_1$ is
\[ 
\frac{(R_\alpha M_1,\beta_1)}{(\beta_1,\beta_1)} = \frac{ (M_1, \beta_1)}{c_{11}} - 2 \frac{ (M_1, \alpha)}{c_{11}} \frac { (\alpha, \beta_1)}{(\alpha, \alpha)}.
\]
This quantity is
\begin{align*}
\frac{ (M_1, \beta_1)}{c_{11}} - 2 \frac{ (M_1, \alpha)}{c_{11}} \frac { (\alpha, \beta_1)}{(\alpha, \alpha)}	&= 1-2 \frac{ (M_1, \alpha)}{c_{11}} \frac { (\alpha, \beta_1)}{(\alpha, \alpha)} \\
		&= 1-2 \frac{2-c_{12}}{4- 2c_{12}}\\
		&= 0,
\end{align*}
by our choices of $\beta_1, \beta_2$.

The component of $R_\alpha M_1$ along $M_2$ is given by
\[ 
\frac{(R_\alpha M_1,\beta_2)}{(\beta_2,\beta_2)} = -2 \frac{(M_1, \alpha)}{c_{22}} \frac{(\alpha, \beta_2)}{(\alpha, \alpha)} = -2 \frac{c_{11}}{c_{22}} \frac{c_{12}-c_{22}}{(4-2c_{12})} = 1.
\]
For $k \geq 3$, we have that the component of $R_\alpha v$ along $M_k$ is 
\[ 
\frac{(R_\alpha v,\beta_k)}{(\beta_k,\beta_k)} = -2 \frac{(M_1, \alpha)}{c_{kk}} \frac{(\alpha, \beta_k)}{(\alpha, \alpha)} + m_k =  -2\frac{2 (c_{1k} - c_{2k})}{c_{kk}(\alpha,\alpha)} +m_k.
\]
We thus choose the integers $m_k$ such that this quantity becomes non-negative.

We see that a suitable positive integral multiple of $v := M_1 + \sum_{k=3}^\ell m_k \cdot M_k$ is mapped to another dominant weight by the reflection $R_\alpha v$.

This proves that, as long as $\ell \geq 3$, there are always two dominant restricted weights $v \neq w$ such that $\lambda:= (v + \delta, v + \delta) = (w + \delta, w + \delta)$, proving that the space of spherical functions $V_v$ and $V_w$ are two different $G$-invariant subspaces of $E_\lambda$.

In the basis $\{M_i\}$, we have that $v = (1,0,m_3, \ldots, m_\ell)$ and $R_\alpha v = (0,1,*)$ so, by means of Lemma \ref{lemma:selfdual}, we can choose the $m_k$ also in order to have the representation $V$ induced by $v$ not isomorphic to the dual of the one $W$ induced by $R_\alpha v$.
\end{proof}

We now move on to the rank-two case. In this case, two dominant restricted weights giving repetitions of eigenvalues need not be related by an isometry of $\lie a^*$. We exhibit explicitly some repetitions.

We refer to the cases in Table \ref{tab:rank2}. Let $\rho = xM_1 + yM_2$. Note that root systems of type $BC_2$ have no outer automorphisms, hence all the spherical representations exhibited are self-dual.

\begin{description}
	\item[AIII (1)] For all $r \geq 4$, we have 
	\[
	(\rho+2\delta,\rho) = 2r x + 4 x^{2} + 2 r y + 4  x y + 2  y^{2} + 4 x,
	\]
	and repetitions are given by $(\ell+3,0)$ and $(\ell-1,6)$ if $r = 2 \ell$ or by $(\ell, 0)$ and $(\ell-2,3)$ if $r = 2 \ell-1$.
	
	\item[A III (2)] We have 
	\[(\rho+2\delta,\rho) =
	4 \, x^{2} + 4 \, x y + 2 \, y^{2} + 10 \, x + 6 \, y.
	\]
	Repetitions are given by the weights $(0,3)$ and $(2,0)$.

	\item[B I] We have
	\[(\rho+2\delta,\rho) =
	4 \, r x + 2 \, x^{2} + 8 \, r y + 4 \, x y + 4 \, y^{2} - 4 \, x - 10 \, y.
	\] 
	For all $r \geq 3$, repetitions are given by $(r+3,0)$ and $(r-3,4)$. For $r=2$, the weight $2 \bar \delta$ is $M_1+M_2$ so we can apply the argument in the proof of Theorem \ref{thmrefl}.
	
	\item[C II (1)] We have
	\[(\rho+2\delta,\rho) =
	4 \, r x + 2 \, x^{2} + 8 \, r y + 4 \, x y + 4 \, y^{2} - 2 \, x - 12 \, y.
	\] If $r > 5$, then the weights $(r+3,0)$ and $(r-6,6)$ give repetitions. For $r=5$ we have the weights $(3,0)$ and $(0,2)$.
	
	\item[C II (2)]  We have
	
	\[(\rho+2\delta,\rho) =
	2 \, x^{2} + 4 \, x y + 4 \, y^{2} + 14 \, x + 20 \, y.
	\]The weights $(0,3)$ and $(3,1)$ give a repetition.
	
	\item[D I (1)] We have
	\[(\rho+2\delta,\rho) =
	2 \, x^{2} + 4 \, x y + 4 \, y^{2} + 6 \, x + 10 \, y.
	\]
	The weights $(0,2)$ and $(3,0)$ give a repetition.
	
	\item[D I (2)] We have
	
	\[(\rho+2\delta,\rho) =
	4 \, r x + 2 \, x^{2} + 8 \, r y + 4 \, x y + 4 \, y^{2} - 6 \, x - 14 \, y.
	\]For all $r \geq 4$ we have the weights $(r,0)$ and $(r-3,2)$.
	
	\item[D III (1)] We have
	\[(\rho+2\delta,\rho) =
	2 \, x^{2} + 4 \, x y + 4 \, y^{2} + 10 \, x + 12 \, y.
	\]
	The weights $(1,5)$ and $(4,3)$ give a repetition.
	
	\item[D III (2)] We have
	\[(\rho+2\delta,\rho) =
	2 \, {\left(2 \, x + y + 11\right)} x + 2 \, {\left(x + y + 7\right)} y.
	\]
	The weights $(0,3)$ and $(2,0)$ give a repetition.
	
	\item[E III] We have
	\[(\rho+2\delta,\rho) =
	2 \, x^{2} + 4 \, x y + 4 \, y^{2} + 22 \, x + 34 \, y.
	\]
	The weights $(1,3)$ and $(4,1)$ give a repetition.
\end{description}

\subsection{Irreducible compact symmetric spaces of type II}
The argument above works for the special case of pairs $(\lie g + \lie g, \lie g)$ where $\lie g$ is diagonal. It is known that the symmetric pair is isometric to $G$ with a bi-invariant metric, so our Theorem \ref{thmrefl} tells us that the only Lie groups $G$ with a $(G\times G)$-simple metric are of rank one.

It follows from Proposition \ref{IrrCond} that a bi-invariant metric on $G$ is never $G$-simple, viewed as a $G$-invariant metric on $G = G/\{1\}$. 
 
\subsection{Riemannian products of normal homogeneous spaces} \label{ProdCROSSs}
Let us now consider Riemannian products of normal homogeneous spaces of the form $M_i = G_i/K_i$, where $G_i$ is a compact semisimple Lie group and $M_i$ is endowed with the Riemannian metric $g_i$ induced from the Cartan-Killing form of $\lie g_i$.

The following fact about eigenspaces of a product metric is well known, see for example \cite{spectre}. 
\begin{thm} Let $(M, g)$ and $(M', g')$ be Riemannian manifolds. Then we have
	\[ 
	E_\lambda(M \times M', g \times g') = \bigoplus_{\substack{\mu \in \Spec(M,g) \\ \nu \in \Spec(M',g') \\ \mu + \nu = \lambda}} E_\mu(M, g) \otimes E_\nu(M', g').
	\]
\end{thm}

Let $G,G'$ be compact Lie groups. We observe that if $V,W$ are irreducible representations of $G,G'$ respectively, then the tensor product $V\otimes W$ is an irreducible representation of $G\times G'$. Let $(M,g)$ (resp. $(M',g')$) be Riemannian manifolds with an isometric action of $G$ (resp. $G'$). Thus, for any $\mu\in\Spec(M,g), \nu\in \Spec(M',g')$ the space $E_\mu(M,g)\otimes E_\nu(M',g')\subset E_{\mu+\nu}(M\times M',g\times g')$ is a $(G\times G')$-subrepresentation.

Given an array of positive real numbers  $\beta = (\beta_1,\dots,\beta_n)\in \R^n_+$,  we consider the metric on $M$ given by $g_\beta = \sum_{i=1}^n \beta_i^{-1} g_i$. We shall refer to a metric of this kind as a \emph{weighted Riemannian product}.

The aim of this subsection is to prove the following.

\begin{thm} \label{thm:generic1}
	Let $(G_i/K_i,g_i)$ be normal $G_i$-simple spherical homogeneous spaces, with $G_i$ semisimple, $i=1,\dots,n$.
	Then a weighted Riemannian product $G/K = \prod_{i=1}^n G_i/K_i$ with metric $g_\beta$ as above is generically $G$-simple. In particular, a generic $G$-invariant Riemannian metric on a reducible symmetric space $G/K$ with factors of rank one, is $G$-simple.
\end{thm}
\begin{proof}
Let $a = (a_1, \ldots, a_n)$ be an array of dominant spherical weights $a_i$ for the pair $(G_i, K_i)$. The weight $a_i$ defines an irreducible representation $V_i=V_{a_i}$ of $G_i$ with $\dim V_{a_i}^{K_i} = 1$ and $V_a = V_{a_1}\otimes\dots\otimes V_{a_n}$ is an irreducible spherical representation of the pair $(G,K)$.

Being each factor endowed with the normal metric, each Casimir operator $\Delta_{g_i}^{V_i}$ acts, by the Freudenthal formula, as $(a_i + 2 \delta_i, a_i)_i$, where $( \cdot, \cdot)_i$ is the Cartan-Killing form of $\lie g_i$ and $\delta_i$ is the half-sum of the positive roots of $\lie g_i$.

Denote by $\lambda_a \in \Q^n$ the array whose $i$-th entry is $(a_i + 2 \delta_i, a_i)_i$. The Casimir of $g_\beta$ acts then on the representation $V_a^K$ by $\lambda_a \cdot \beta$ times the identity, where $\lambda_a\cdot \beta  = \sum_{i=1}^n \beta_i(a_i + 2 \delta_i, a_i)_i$.

Then $(M, g_\beta)$ if $G$-simple if, and only if 
\begin{equation}
B:= \R^n_+ \setminus \bigcup_{a \neq a'} (\lambda_a - \lambda_{a'})^\perp \neq \emptyset.
\end{equation}

The space $B$ is the positive orthant in  $\R^n$ with a countable set of hyperplanes removed. It is thus empty if, and only if, one of these hyperplanes is the whole space, i.e. if there exist $a \neq a'$ such that $\lambda_a = \lambda_{a'}$. This means that there exists an $i\in\{1,\dots,n\}$ such that $a_i\neq a_i'$ but $(a_i + 2 \delta_i, a_i)_i=(a'_i + 2 \delta_i, a'_i)_i$. In other words,  there exists an $i\in\{1,\dots,n\}$ such that  $M_i$ is not $G_i$-simple.
\end{proof}

\section{Towards the case of non-symmetric homogeneous spaces} \label{sec:nonsym}

For simple compact connected $G$, Kr\"amer \cite{kraemer} provides a classification of possible subgroups $K$ along with a set of dominant weights whose integral non-negative combinations give all spherical representations.

Spherical pairs are a generalization of symmetric pairs, but the theory of restricted root systems does not in general work. However, they are weakly symmetric, as proven in \cite{weaklysymm}.

The generators of the dominant spherical weights need not form a free semigroup and the machinery provided by the theory of restricted weights and roots that we used to prove Theorem \ref{thmrefl}, in particular the facts in Subsection~\ref{sec:delta}, do not hold for non-symmetric pairs (cf. \cite{GinGoo}).

Moreover, the spherical pairs in Kr\"amer's list are not isotropy irreducible, so the study of the quadratic form given by Freudenthal's formula only covers the case of the normal metric. The only two examples of  spherical non-symmetric isotropy irreducible homogeneous spaces  have been treated in \cite{spectre} and are described below.

\begin{example} \label{ex:G2SU3}
	The group $\SU(3)$ is embedded in the standard way in $\G_2$ and the isotropy representation on $\R^6$ is the classical one of $\SU(3)$ of $\C^3$. It is transitive on the unit sphere, so we apply the result in \cite{spectre} to conclude that the spectrum of the homogeneous nearly K\"ahler $S^6 = \G_2 / \SU(3)$ is $\G_2$-simple. This pair is spherical and not symmetric.
\end{example}

\begin{example} \label{ex:SO7G2}
	The standard inclusion $\G_2 \subset \SO(7)$ gives rise to a spherical pair. The homogeneous space $M = \SO(7) / \G_2$  has a $\SO(7)$-simple spectrum with respect to the normal metric, since $\G_2$ acts transitively on the  unit $7$-sphere. \end{example}

In this section we discuss two examples. The first is a class of circle bundles that submerge Hermitian symmetric spaces, and the second is a non-spherical example. So, as a by-product, we obtain the following, in contrast with the case of $G$-irreducibility of \emph{complex} Laplacian eigenspaces.
\begin{prop}
Let $G/K$ be a Riemannian homogeneous space with $G$-simple spectrum (over the reals). Then $K$ need not be spherical in $G$.
\end{prop}

\subsection{Circle bundles over Hermitian symmetric spaces} \label{sec:bundles}

In Kr\"amer's list, there is a class of spherical homogeneous spaces that are circle bundles over Hermitian symmetric spaces, namely
\begin{enumerate}[label={B \arabic*}.,ref={B \arabic*}]
	\item \label{bdlAIII} $\SU(n+m)/(\SU(m) \times \SU(n)) \to \SU(n+m)/S(\U(n) \times \U(m))$ where the base is of Cartan type AIII;
	\item \label{bdlDIII} $\SO(2n)/\SU(n) \to \SO(2n)/\U(n)$, for odd $n \geq 3$, where the base is of Cartan type DIII;
	\item \label{bdlEIII} $\E_6/\D_5 \to \E_6/(\U(1) \times \D_5)$ on a base of Cartan type EIII.
\end{enumerate}

Let us unify the notation by writing $M=G/K$ for the total space of the circle bundle,  $N=G/H$ for the symmetric base and $H/K$ for the fiber.

From \cite{kraemer_factors, kraemer}, we can see that, in these cases, the isotropy representation of $K$ splits as 
\begin{equation} \label{eq:splitting}
\lie g = \lie k \oplus \R H_0 \oplus \lie m,
\end{equation}
where $H_0$ generates the vertical subbundle of $G/K \to G/H$ and, if $\lie a_{\lie k}$ is a maximal torus of $\lie k$, then $\lie a_{\lie k} \oplus \R H_0$ is a maximal torus of $\lie g$.

We may scale any invariant metric on $M$ to make these bundles Riemannian submersions onto the symmetric space $N$. Moreover, the fiber $K/H$ is a totally geodesic circle, as one can compute from the Levi-Civita connection of the invariant metric, see e.g. \cite{besse}.

Using the well-known relations between the Laplacians of $M$ and $N$, see \cite{onishchik2012lie}, we can see that every eigenfunction of $N$ is also an eigenfunction on $M$ with the same eigenvalue. Moreover, both eigenspaces are real $G$-modules. We then have the following.
\begin{lemma}
	If $N$ is not $G$-simple, then neither is $M$.
\end{lemma}

Indeed, the spectrum of $N$ is contained in the spectrum of $M$, and for all $\mu \in \Spec(N)$, we have that the eigenspace $E_\mu(N)$ is a $G$-invariant subspace of $E_\mu(M)$.

Since our base spaces $N$ are irreducible symmetric spaces, it is clear from Theorem \ref{thmrefl} when they are simple, namely only when $m=1$ for case \eqref{bdlAIII}, $n=3$ for \eqref{bdlDIII} and never for \eqref{bdlEIII}.

In case \eqref{bdlAIII} we have the Hopf fibration of $M = S^{2n+1}= \SU(n+1)/\SU(n)$ over $\CP^n$ and in case \eqref{bdlDIII} we have $N = \SO(6)/\U(3)$.

We denote by $\hat G_K$ the set of maximal weights corresponding to irreducible representations of $G$ spherical with respect to $K$. Generators $M_i$ for such semigroups are given in Kr\"amer's table. Note that there are some $M_i$ that do not vanish on $\lie a_{\lie k}$. 

Since the two summands in \eqref{eq:splitting} are $\Ad(K)$-inequivalent, any $G$-invariant Riemannian metric $g_\gamma$ on $G/K$ is parametrized by two positive real numbers $\gamma_1, \gamma_2$ and equals $\gamma_1^{-1}$  (resp. $\gamma_2^{-1}$) times the Killing form restricted to $\R\cdot H_0$ (resp. $\lie m$). So we identify $\Sym_K^+(\lie m) \simeq \R_+^2$ and $\Sym_K(\lie m) \simeq \R^2$.

Let $X_1, \ldots X_k$ be a basis of $\lie m$ orthonormal with respect to the Killing form. 

For $\gamma = (\gamma_1, \gamma_2)$, the Casimir operator $C_\gamma$ of $g_\gamma$ is
\begin{equation}  \label{eq:casimirbdl}
\begin{split}
C_\gamma 	&= \gamma_1 H_0^2 + \gamma_2 \sum_k X_k^2 \\ 
&= (\gamma_1 - \gamma_2) H_0^2 + \gamma_2 C,
\end{split}
\end{equation}
where $C$ is the Casimir operator of $\lie g$. 

Let $V\in \hat G_K$. Since $(G,K)$ is a spherical pair, the fixed point set $V^K$ is a line, which is mapped to itself by $\Delta_\gamma^{V^K} = \Delta_{g_\gamma}^{V^K}$. Let $\lambda$ be the highest weight of $V$ and let $v\in V$ be a generator of $V_\lambda^K$. We have
\[
-\Delta_\gamma^{V^K}v =  (\gamma_1 - \gamma_2) \alpha_\lambda v + \gamma_2 (\lambda + 2 \delta, \lambda)v,
\]
since the Casimir $C$ are given by Freudenthal's formula and $\alpha_\lambda v := (H_0)^2v$.

So the system to consider to check property \eqref{irrcondA} of Proposition \ref{IrrCond} is
\begin{equation}
\label{sys:alpha}
\begin{cases}
\alpha_\lambda = \alpha_{\lambda'}\\
(\lambda + 2 \delta, \lambda) = (\lambda' + 2 \delta, \lambda')
\end{cases}
\end{equation}
for two different dominant spherical weights $\lambda$ and $\lambda'$.

We consider Case \eqref{bdlAIII}. From Kr\"amer's list, we have that such a representation of $\SU(n+1)$ is given by the highest weight $p \pi_1 + q \pi_n$ with $p,q \geq 0$. This representation is known to be the space $\call H_{p,q}(\C^{n+1})$ of harmonic complex polynomials of $n+1$ variables that are homogeneous of degree $p$ in the $z_i$ and of degree $q$ in the $\bar z_i$, see e.g. \cite{sepanski}.

Using the explicit $\SU(n+1)$-invariant projection map $P$ from \cite[Thm.~1]{projharm} from the space of all polynomials of bidegree $(p,q)$ to $V=\call H_{p,q}(\C^{n+1})$, one can see that the space $V^K$ is spanned by $v= P(z_{n+1}^p \bar z_{n+1}^q)$ and that $Hv = in(-p+q)v$, so $\alpha_\lambda = -n^2(-p+q)^2$.

The Freudenthal formula can be computed by means of the inverse of the Cartan matrix (see e.g. \cite[p.~295]{onishchik2012lie}), so the system \eqref{sys:alpha} is equivalent to
\begin{equation}
\begin{cases}
(-p+q)^2 = (-p' + q')^2\\
n(p^2 + q^2) + 2pq + n(p+q) = n(p'^2 + q'^2) + 2p'q'+ n(p'+q').
\end{cases}
\end{equation}

Substituting the first equation in the second, we obtain
\begin{equation}
\begin{cases}
(-p+q)^2 = (-p' + q')^2\\
2(n+1)pq + n(p+q) = 2(n+1)p'q'+ n(p'+q').
\end{cases}
\end{equation}

With the substitution $x = 2(n+1)p + n$ and $y=2(n+1)q+n$ and their analog with the prime, the system is equivalent to
\begin{equation}
\begin{cases}
x^2 + y^2 = x'^2 + y'^2 \\
xy = x'y'.
\end{cases}
\end{equation}
The problem is, in turn, equivalent to finding repetitions in the first quadrant of the quadratic form $x^2 + y^2$, so we know that the only repetition can be obtained by a swap of $x,y$.

\begin{prop} \label{prop:hopf}
	Let $n>1$. A generic $\SU(n+1)$-invariant Riemannian metric on the total space of the Hopf fibration is $\SU(n+1)$-simple.
\end{prop}

\begin{proof}
	Since $(G,K)$ is spherical, i.e. for all $V \in \hat G_K$ the fixed point space $V^K$ is one-dimensional, the polynomial $p_V(\gamma)(t) = t- \Delta_\gamma^{V^K}$ is linear in $t$. Thus, condition \eqref{irrcondresB} and \eqref{irrcondresC} of Proposition \ref{prop:resultants} are automatically satisfied. 
	
	Assume that no $G$-invariant metric is $G$-simple on $G/K$, so condition \eqref{irrcondresA} of Proposition \ref{prop:resultants} is violated, i.e. there exist $V, W \in \hat G_K$ such that $V \not\cong W$ and $V^* \not \cong W$ and the polynomial
	\[
	\res \circ (p_V, p_W) \colon \Sym_K(\lie m) \simeq \R^2 \to \C
	\]
	is the zero polynomial. 
	
	The above computation shows that if such polynomial is zero on $\Sym_K^+(\lie m)$, then $V \cong W$ or $V^* \cong W$, a contradiction.
\end{proof}
For $n=1$, the existence of generic $\SU(2)$-simple metrics on $\SU(2)$ is obtained by Schueth \cite{schueth} but not by stretching the fibers of the Berger sphere, as the conditions of Proposition \ref{IrrCond} are not satisfied.

In case \eqref{bdlDIII} for $n=3$, note that $\SO(6) = \SU(4)/\{ \pm I\}$, so this induces a double covering $\SU(4)/\SU(3) \to \SO(6)/\SU(3)$. We can then apply Lemma \ref{lemma:covering}, since the total space is treated in case \eqref{bdlAIII}.

\subsection{A quotient of $\RP^3$} \label{sec:SU2F}
We now exhibit an example of a non-spherical $G$-simple homogeneous space $G/K$. The low-dimensionality of this example allows us to examine its representations and the action of the Casimir operator.
Let $G = \SU(2)$ and let $F$ be the finite subgroup generated by
\[
\sigma = \begin{pmatrix}
e^\frac{i \pi}{3} & 0 \\ 0 & e^{-\frac{i \pi}{3}}
\end{pmatrix}; \qquad
\tau = \begin{pmatrix}
0 & i \\ i & 0
\end{pmatrix}.
\]

The homogeneous space $\SU(2)/F$ has been studied in \cite{art:bedulli_stab}.

Let 
\[
H = \frac{1}{\sqrt 8} \begin{pmatrix}
i & 0 \\ 0 & -i
\end{pmatrix};
X = \frac{1}{\sqrt 8} \begin{pmatrix}
0 & 1 \\ -1 & 0
\end{pmatrix};
Y =  \frac{1}{\sqrt 8} \begin{pmatrix}
0 & i \\ i & 0
\end{pmatrix}
\]
be an orthonormal basis of $\lie{su}(2)$ with respect to the Cartan-Killing form.

The isotropy representation of $F$ splits as irreducible $F$-modules as
\[
\lie{su}(2) = \R \cdot H \oplus \Span \{ X, Y \},
\]
hence all possible $\SU(2)$-invariant metrics are parameterized by two positive real numbers.

All the complex irreducible representations of $\SU(2)$ are parameterized by a positive integer $k$ and are given by $V_k = S^k \C^2$, that can be expressed as the space of symmetric polynomials over $\C$ of degree $k$ in two variables $z_1, z_2$. Denote by $v_\ell$ the polynomial $z_1^\ell z_2^{k-\ell}$. It is known that $V_k$ is of real type if $k$ is even and of quaternionic type if $k$ is odd.

We want to prove the existence of metrics that satisfy the conditions in Proposition \ref{IrrCond}.

The pair $(\SU(2), F)$ is not spherical and its spherical representations were computed in \cite{art:bedulli_stab}.
They verify that all spherical representations $V_k$ need to have $k$ even, so they are all of real type and condition \eqref{irrcondC} is always empty. We focus then on conditions \eqref{irrcondA} and \eqref{irrcondB}.

Any $\SU(2)$-invariant Riemannian metric on $M$ is induced by a scalar product on $\lie{su}(2)$ that equals $a^{-1}$ times the Killing form on $\R \cdot H$ and $b^{-1}$ times the Killing form on $\lie m_2 := \Span\{X, Y \}$, for some positive reals $a, b$.

The Casimir operator is then
\begin{align}
C_{a,b} 	&= \frac{a}{8} H^2 + \frac{b}{8} (X^2 + Y^2) \\
		&= \frac{b}{8} (H^2 + X^2 + Y^2) + \frac{a-b}{8} H^2 \\
		&= b C + \frac{a-b}{8} H^2,
\end{align}
where $C$ is the standard Casimir operator of $\lie{su}(2)$.

In \cite{art:bedulli_stab} the authors explicitly compute a basis of $V_k^F$ for all $k$. They also compute that the action of $H^2$ on $v_\ell$, given by
\[
H^2 \cdot v_\ell = - (2 \ell - k)^2 v_\ell.
\]

Moreover, the Casimir $C$ acts on $V_k$ as $\frac 1 8 k (k+2) \cdot \id$. Hence, the our Casimir $C_{a,b}$ acts on $V_k^F$ by
\begin{equation}
C_{a,b} v_\ell = \biggl [ b k(k+2) - \frac{a-b}{8} (2 \ell - k)^2 \biggr ] v_\ell.
\end{equation}

From this expression and the explicit generators of $V_k^F$ determined in \cite{art:bedulli_stab}, we can say that on every $V_k^F$ the operator $C_{a,b}$ acts with distinct eigenvalues for all $a,b$. So condition \eqref{irrcondB} in Proposition \ref{IrrCond} is satisfied.

To verify condition \eqref{irrcondA} for appropriate $a,b$, we can apply a genericity argument similar to the one in Subsection \ref{ProdCROSSs}.
Define 
\[
\lambda_{k,\ell} = \left(-\frac{(2\ell-k)^2}{8}, k(k+2) + \frac{(2\ell-k)^2}{8} \right) \in \R^2
\] 
Then for $v_\ell\in V^F_k$ as above
\[
C_{a,b}v_\ell = \lambda_{k,\ell}\cdot \begin{pmatrix}
a \\ b
\end{pmatrix}.
\]
Thus, if for any choice of $a,b$ the above metric violates condition \eqref{irrcondA} we would have $\lambda_{k\ell} = \lambda_{k'\ell'}$ for some $k,\ell,k',\ell'$. But looking at the explicit form of $\lambda_{k\ell}$ we see that this implies $k=k'$ and since we know that $C_{a,b}$ acts on $V_k^F$ with distinct eigenvalues for any choice of $a,b$ it follows that $\ell=\ell'$.

So, we can conclude that $\SU(2)/F$ admits $\SU(2)$-simple metrics.
\bibliography{../../allbib/allbib}
\bibliographystyle{amsplain}

\end{document}